\documentclass[final,leqno]{siamltex704}
\usepackage{hyperref}
\usepackage{amsmath}
\usepackage{graphicx}
\usepackage{array}
\usepackage{lineno}
\usepackage[notcite,notref]{showkeys}
\usepackage{tikz}
\usepackage{amsfonts,amssymb}
\usepackage{dsfont}
\usepackage{pifont}
\usepackage{color}
\usepackage{subfigure}
\newtheorem{remark}{Remark}[section]

\def\F{{\mathcal F}}
\def\E{{\mathcal E}}
\def\S{{\mathcal Q_{h}}}

\def\O{\Omega}
\def\pa{\partial}
\def\W{{\mathcal W}}
\def\3bar{{|\hspace{-.02in}|\hspace{-.02in}|}}

\def\bn{{\bf n}}
\newtheorem{WG}{Weak Galerkin Algorithm}

\begin{document}

\setlength{\parindent}{0.25in} \setlength{\parskip}{0.08in}

\title{
An Extension of the Morley Element on General Polytopal Partitions Using Weak Galerkin Methods
}

\author{
Dan Li\thanks{Jiangsu Key Laboratory for NSLSCS, School of Mathematical Sciences,  Nanjing Normal University, Nanjing 210023, China (danlimath@163.com). {The research of Dan Li was supported by Jiangsu Funding Program for Excellent
Postdoctoral Talent and} China National Natural Science Foundation Grant (No. 12071227).} \and
Chunmei Wang \thanks{Department of Mathematics, University of Florida, Gainesville, FL 32611 (chunmei.wang@ufl.edu). The research of Chunmei Wang was partially supported by National Science Foundation Grants DMS-2136380 and DMS-2206332.}
\and
Junping Wang\thanks{Division of Mathematical Sciences, National Science Foundation, Alexandria, VA 22314 (jwang@nsf.gov). The research of Junping Wang was supported by the NSF IR/D program, while working at National Science Foundation. However, any opinion, finding, and conclusions or recommendations expressed in this material are those of the author and do not necessarily reflect the views of the National Science Foundation.}}

\maketitle

\begin{abstract}
This paper introduces an extension of the well-known Morley element for the biharmonic equation, extending its application from triangular elements to general polytopal elements using the weak Galerkin finite element methods. By leveraging the Schur complement of the weak Galerkin method, this extension not only preserves the same degrees of freedom as the Morley element on triangular elements but also expands its applicability to general polytopal elements. The numerical scheme is devised by locally constructing weak tangential derivatives and weak second-order partial derivatives. Error estimates for the numerical approximation are established in both the energy norm and the $L^2$ norm. A series of numerical experiments are conducted to validate the theoretical developments.
\end{abstract}



\section{Introduction}
This paper focuses on the development of the Morley element for the biharmonic equation using the weak Galerkin (WG) method. For simplicity, we consider the following biharmonic model equation
\begin{equation}\label{model-problem}
\begin{split}
\Delta^2u&=f,\quad\mbox{in}~~\O, \\
        u&=g,\quad\mbox{on}~~\pa\O,\\
\frac{\pa u}{\pa\textbf{n}}&=\nu,\quad\mbox{on}~~\pa\O,
\end{split}
\end{equation}
where $\Omega$ is a bounded polytopal domain in $\mathbb{R}^d (d=2,3)$ and the vector $\textbf{n}$ is an unit outward normal direction to $\pa\O$.

A weak formulation of \eqref{model-problem} reads: Find $u\in H^2(\O)$ such that $u|_{\pa\O}=g$ and $\frac{\pa u}{\pa\textbf{n}}|_{\pa\O}=\nu$ satisfying
\begin{equation}\label{weakform}
\sum_{i,j=1}^d(\pa_{ij}^2u,\pa_{ij}^2v)=(f,v),~~~~\forall v\in H_0^2(\O),
\end{equation}
where $H_0^2(\O)=\{v\in H^2(\O):v|_{\pa\O}=0, \nabla v |_{\pa\O}=\textbf{0}\}$.

The construction of $C^1$ continuous finite elements often necessitates higher order polynomial functions, which can pose challenges in numerical implementation. To address this issue, several nonconforming finite element methods have been proposed. Among them, the Morley element \cite{LM1968} is a well-known nonconforming finite element that minimizes the degrees of freedom but is limited to triangular partitions. In subsequent works \cite{rusa1988,WX2007-3,WX2006}, the Morley element was extended to higher dimensions. Subsequently, \cite{WXH2006,WX2007-2,PS2013,ZZCM2018,ZCM2016} explored the development of the Morley element for general polytopal partitions. In addition to these studies, numerous numerical methods have been developed to solve the biharmonic equation, including discontinuous Galerkin methods \cite{CDG2009,PHML2002,MSB2007,YZ2019}, virtual element methods \cite{AMV2018,CH2020}, and weak Galerkin methods \cite{BGZ2020,SIAM_YZ-2020, MWYZ2014,MWY2014,WW_bihar-2014,WW_HWG-2015}. The weak Galerkin (WG) method, first introduced by Junping Wang and Xiu Ye for second-order elliptic problems \cite{wy}, provides a natural extension of the classical finite element method through a relaxed regularity of the approximating functions.  This novelty provides a high flexibility in numerical approximations with any needed accuracy and  
mesh generation being general polygonal or polyhedral partitions.  To the best of our knowledge, no existing numerical method combines the advantages of minimal degrees of freedom and applicability to general polytopal partitions.

The objective of this paper is to present an extension of the Morley element to general polytopal meshes utilizing the weak Galerkin (WG) method. Drawing inspiration from the de Rham complexes \cite{WWYZ-JCAM2021}, we propose a modification to the original weak finite element space by incorporating additional approximating functions defined on the $(d-2)$-dimensional sub-polytopes of $d$-dimensional polytopal elements. Through the utilization of the Schur complement within the WG method, this innovative approach introduces NE+NF degrees of freedom on general polytopal partitions, where NE and NF represent the numbers of $(d-2)$-dimensional sub-polytopes and $(d-1)$-dimensional sub-polytopes of $d$-dimensional polytopal elements, respectively. The resulting numerical algorithm is designed based on locally constructed weak tangential derivatives and weak second-order partial derivatives. Additionally, we establish error estimates for the numerical approximation in both the energy norm and the $L^2$ norm.

This paper makes several key contributions. Firstly, compared to the well-known Morley element, the proposed WG method allows for the utilization of the local least degrees of freedom on general polytopal elements. This extension enhances the versatility of the Morley element and expands its applicability to a wider range of problems. Secondly, in contrast to existing results on weak Galerkin methods, we introduce a novel technique within the framework of weak Galerkin, which effectively reduces the number of unknowns. This reduction in unknowns improves computational efficiency without sacrificing accuracy. Lastly, our numerical method can be applied to address various partial differential problems, including model problems with weak formulations based on the Hessian operator. This broad applicability demonstrates the effectiveness and potential of our proposed approach in tackling diverse problem domains.

This paper is structured as follows. Section \ref{Section:WeakHessian} provides a review of the definitions of the weak tangential derivative and the weak second-order partial derivatives. In Section \ref{Section:numerical scheme}, we present the weak Galerkin scheme and introduce the concept of its Schur complement. The existence and uniqueness of the solution are investigated in Section \ref{Section:seu}. An error equation for the proposed weak Galerkin scheme is derived in Section \ref{Section:error-equation}. Section \ref{technique-estimate} focuses on deriving technical results to support the analysis. The error estimates for the numerical approximation in the energy norm and the $L^2$ norm are established in Section \ref{Section:EE}. Finally, in Section \ref{Section:NE}, we present a series of numerical results to validate the theoretical developments presented in the preceding sections.

The standard notations are adopted throughout this paper. Let $D$ be any open bounded domain with Lipschitz continuous boundary in $\mathbb{R}^d$. We use $(\cdot,\cdot)_{s,D}$, $|\cdot|_{s,D}$ and $\|\cdot\|_{s,D}$ to denote the inner product, semi-norm and norm in the Sobolev space $H^s(D)$ for any integer $s\geq0$, respectively. For simplicity, the subscript $D$ is  dropped from the notations of the inner product and norm when the domain $D$ is chosen as $D=\O$. For the case of $s=0$, the notations $(\cdot,\cdot)_{0,D}$, $|\cdot|_{0,D}$ and $\|\cdot\|_{0,D}$ are simplified as $(\cdot,\cdot)_D$, $|\cdot|_D$ and $\|\cdot\|_D$, respectively. The notation ``$A\lesssim B$'' refers to the inequality ``$A\leq CB$'' where $C$ presents a generic constant independent of the meshsize or the functions appearing in the inequality.

\section{Discrete weak partial derivatives}\label{Section:WeakHessian}

Let ${\cal T}_h$ be a polygonal or polyhedral partition of the domain $\O$ that is shape regular as specified in \cite{WY-ellip_MC2014}. For each $d$-dimensional polytopal element $T\in{\cal T}_h$, let $\pa T$ be the boundary of $T$ that is the set of $(d-1)$-dimensional polytopal elements denoted by $\F$ (called ``face" for convenience). For each face $\F\subset\pa T$, let $\pa\F$ be the boundary of $\F$ that is the set of $(d-2)$-dimensional polytopal elements denoted by $e$ (called ``edge" for convenience). Let $\F_h$ be the set of all faces in ${\cal T}_h$ and denote by $\F_h^0= \F_h \setminus \partial\Omega$ the set of all interior faces, respectively. Similarly, let $\mathcal{E}_h$ be the set of all edges in ${\cal T}_h$ and denote by $\mathcal{E}_h^0= \mathcal{E}_h \setminus \partial\Omega$ the set of all interior edges, respectively. Moreover, we denote by $h_T$ the diameter of $T\in {\cal T}_h$ and $h=\max_{T\in{\cal T}_h}h_T$ the meshsize of ${\cal T}_h$, respectively. For any given integer $r\geq 0$, let $P_r(T)$ and $P_r(\pa T)$ be the sets of polynomials  on $T$ and $\pa T$ with degrees no greater than $r$, respectively.

For each element $T\in {\cal T}_h$, by a weak function on $T$ we mean a triplet $v=\{v_0,v_b,v_n\bn_{f}\}$, where $v_0$ and $v_b$ are intended for the values of $v$ in the interior of $T$ and on the edge $e$ respectively, and $v_n$ is used to represent the normal component of the gradient of $v$ on the face $\F$ along the direction $\bn_{f}$ being the unit outward normal vector to $\F$. Note that $v_b$ is defined on edge $e$ that is different from the case when $v_b$ is defined on face $\F$ as proposed in \cite{WW_bihar-2014,WW_HWG-2015,ww1}.

We introduce the local discrete space of the weak functions given by
$$
V(T)=\{v=\{v_0,v_b,v_n\bn_{f}\}:v_0\in P_2(T),v_b\in P_0(e),v_n\in P_0(\F),~\F\subset\pa T, e\subset\pa\F\}.
$$

For each face $\F\in \F_h$, denote by ${{\W}}_0(\F)$ the finite element space consisting of constant vector-valued functions tangential to $\F$ given by
$$
{{\W}}_0(\F)=\{\pmb{\psi}:~ \pmb{\psi}\in [P_0(\F)]^d,\ \pmb{\psi}\cdot\bn_{f}=0\}.
$$

\begin{definition}\cite{WWYZ-JCAM2021}(Discrete weak tangential derivative)\label{weak tangential derivativee}
The discrete weak tangential derivative operator, denoted by $\nabla_{w,0,\F}$, is defined as the unique vector-valued polynomial  $\nabla_{w,0,\F}v\in{{\W}}_0(\F)$ for any $v\in V(T)$ satisfying the following equation:
\begin{eqnarray}\label{weak tangential derivative}
 \langle\nabla_{w,0,\F}v,\pmb{\psi}\times\bn_{f}\rangle_{{\F}}
=\langle v_b,\pmb{\psi}\cdot\pmb{\tau}\rangle_{\pa\F},\quad\forall\pmb{\psi}\in{{\W}}_0(\F).
\end{eqnarray}
Here, $\pmb{\tau}$ denotes the unit vector tangential to $\pa\F$ that is chosen such that $\pmb{\tau}$ and $\pmb{n}_{f}$ obey the right-hand rule.
\end{definition}

From the normal derivative $v_n$ and the discrete weak tangential derivative $\nabla_{w,0,\F}v$, the discrete weak gradient of $v$ on face ${\F}$, denoted by $\pmb{v_g}$, can be decomposed into its normal and tangential components; i.e.,
\begin{equation}\label{vgi}
   \pmb{v_g}=v_n\bn_{f}+\nabla_{w,0,\F}v.
\end{equation}

\begin{definition}\cite{WW_bihar-2014}(Discrete weak second order partial derivative)\label{discrete weak partial derivetivee}
The discrete weak second order partial derivative operator, denoted by $\pa_{ij,w,0,T}^2$, is defined as the unique polynomial  $\pa_{ij,w,0,T}^2v\in P_0(T)$ for any $v\in V(T)$ satisfying the following equation:
\begin{eqnarray}\label{discrete weak partial derivetive}
(\pa_{ij,w,0,T}^2v,\varphi)_T=\langle v_{gi}, \varphi n_j\rangle_{\pa T},~~~ \forall\varphi\in P_0(T).
\end{eqnarray}
Here, $v_{gi}$ is $i$-th component of the vector $\pmb{v_g}$ given by \eqref{vgi} and $\bn=(n_1,\cdots, n_d)$ is the unit outward normal direction to $\pa T$, respectively.
\end{definition}

Applying the integration by parts to $(v_0,\pa_{ji}^2\varphi)_T$ and combining \eqref{discrete weak partial derivetive} yield
\begin{equation}\label{error equation-13}
(\pa_{ij,w,0,T}^2v,\varphi)_T
=(\pa_{ij}^2v_0,\varphi)_T-\langle(\pa_iv_0-v_{gi})n_j,\varphi\rangle_{\pa T}
\end{equation}
for any $\varphi\in P_{0}(T)$.

\begin{remark}
 Note that in Definitions \ref{weak tangential derivativee}-\ref{discrete weak partial derivetivee}, the discrete weak tangential derivative and the discrete weak second order partial derivative are discretized  by the lowest order polynomial functions in ${{\W}}_0(\F)$ and $P_0(T)$, respectively. When it comes to the higher order polynomial approximations in  ${{\W}}_r(\F)$ and $P_r(T)$ for an integer $r\geq1$, Definitions \ref{weak tangential derivativee}-\ref{discrete weak partial derivetivee} need to be redesigned accordingly.
 \end{remark}

\section{Weak Galerkin schemes}\label{Section:numerical scheme}
By patching the local finite element $V(T)$ over all the elements $T\in{\cal T}_h$ through the common values on the interior edges $\mathcal{E}_h^0$ for $v_b$ and the interior faces $\F_h^0$ for $v_n\bn_{f}$, we obtain a global weak finite element space $V_h$ as follows
 $$
 V_h=\{v=\{v_0,v_b,v_n\bn_{f}\}:v|_T\in V(T),~T\in{\cal T}_h\}.
 $$
Denote by $V_h^0$ the subspace of $V_h$ with homogeneous boundary conditions for $v_b$ and $v_n$ on $\pa\O$ given by
$$
V_h^0=\{v:v\in V_h, v_b|_e=0,v_n|_{\F}=0, e\subset\pa \O, \F\subset\pa\O\}.
$$

For simplicity of notation, the discrete  weak tangential derivative $\nabla_{w,0,\F}v$ defined by \eqref{weak tangential derivative} and the discrete weak second order partial derivative $\pa_{ij,w,0,T}^2v$ computed by \eqref{discrete weak partial derivetive} are simplified as follows
$$
(\nabla_{w,\F}v)|_T=\nabla_{w,0,\F}(v|_T),\quad(\pa^2_{ij,w}v)|_T=\pa_{ij,w,0,T}^2(v|_T),\quad v\in V_h.
$$

For any $\sigma,v\in V_h$, let us introduce the following bilinear forms:
\begin{equation*}\label{stabilizer}
\begin{split}
(\pa^2_w\sigma,\pa^2_wv)_{{\cal T}_h}=&\sum_{T\in{\cal T}_h}\sum_{i,j=1}^d(\pa_{ij,w}^2\sigma,\pa_{ij,w}^2v)_T,\\
s(\sigma,v)=&\sum_{T\in{\cal T}_h}h_T^{-2}\langle Q_b\sigma_0-\sigma_b,Q_bv_0-v_b\rangle_{\pa\F}\\
       &+\sum_{T\in{\cal T}_h}h_T^{-1}\langle Q_n(\nabla \sigma_0)\cdot\bn_{f}-\sigma_n, Q_n(\nabla v_0)\cdot\bn_{f}-v_n\rangle_{\pa T},\\
a(\sigma,v)=&(\pa^2_w\sigma,\pa^2_wv)_{{\cal T}_h}+s(\sigma,v),
\end{split}
\end{equation*}
where $Q_b$ and $ Q_n$ represent the usual $L^2$ projection operators onto $P_0(e)$ and $P_0(\F)$, respectively.

\begin{WG}
A numerical approximation for \eqref{weakform} is as follows: Find $u_h=\{u_0,u_b,u_n\bn_{f}\}\in V_h$ such that $u_b=Q_bg$  on $e\subset\pa\O$ and $u_n= Q_n\nu$ on $\F\subset\pa\O$ satisfying
\begin{equation}\label{WG-scheme}
a(u_h,v)=(f,v_0),\qquad\forall v\in V_h^0.
\end{equation}
\end{WG}

As an effective approach, the Schur complement technique \cite{Eff-MWY2017,MWYS2015} could be incorporated into the WG scheme \eqref{WG-scheme} to reduce the number of the unknowns. More precisely, a numerical approximation of the Schur complement for \eqref{WG-scheme} is to find $u_h=\{D(u_b,u_n,f),u_b,u_n\bn_{f}\}\in V_h$ satisfying $u_b=Q_bg$ on $e\subset\pa \O$, $u_n= Q_n\nu$ on $\F\subset\pa\O$ and the following equation:
\begin{equation}\label{WG-schemee}
a(\{D(u_b,u_n,f),u_b,u_n\bn_{f}\},v)=0,\,\qquad\forall v=\{0,v_b,v_n\bn_{f}\}\in V_h^0,
\end{equation}
where $u_0=D(u_b,u_n,f)$ can be obtained by solving the equation as follows
\begin{equation}\label{WG-schemeee}
a(\{u_0,u_b,u_n\bn_{f}\},v)=(f,v_0),\quad\quad~\forall v=\{v_0,0,\textbf{0}\}\in V_h^0.
\end{equation}

\begin{remark}
The Schur complement of WG scheme \eqref{WG-schemee}-\eqref{WG-schemeee} and the WG scheme \eqref{WG-scheme} have the same numerical approximation, for which the similar proof can be found in \cite{Eff-MWY2017}. The degrees of freedom of \eqref{WG-schemee}-\eqref{WG-schemeee} are shown in Figure \ref{polygonal-element} for two polygonal  elements: a triangle and a pentagon.
\begin{figure}[htp]
\begin{center}
\begin{tikzpicture}
\coordinate (B1) at (-5.5,1.2); \filldraw[green] (B1) circle(0.08);
\coordinate (B2) at (-6.5,-1.5);  \filldraw[green] (B2)circle(0.08);
\coordinate (B3) at (-3.5,-1); \filldraw[green] (B3) circle(0.08);
\draw node[above] at (B1) {$u_{b1}$}; \draw node[left] at (B2){$u_{b2}$}; \draw node[right] at (B3){$u_{b3}$};
\coordinate (MM1) at (-5,-1.25); \coordinate (MM2) at (-4.5,0.1); \coordinate (MM3) at (-6,-0.2);
\coordinate (MM1end) at (-4.85,-2.2); \coordinate (MM2end) at (-3.8,0.7); \coordinate (MM3end) at (-6.9,0.15);
\coordinate (center) at (-5.16,-0.43);
\coordinate (ne1) at (-4.9,-1.9);\coordinate (ne2) at (-4.3,0.4);\coordinate (ne3) at (-6.4,0);
\draw node[right] at (ne1) {$u_{n1}$};\draw node[above] at (ne2) {$u_{n2}$};\draw node[above] at (ne3) {$u_{n3}$};
\draw node at (center) {T};
 \draw (B1)--(B2)--(B3)--cycle;
\filldraw[red] (MM1) circle(0.08);\filldraw[red] (MM2) circle(0.08);\filldraw[red] (MM3) circle(0.08);
\draw[->] (MM1)--(MM1end);\draw[->] (MM2)--(MM2end);\draw[->] (MM3)--(MM3end);
\coordinate (A1) at (-0.2,1.5); \filldraw[green] (A1) circle(0.08);
\coordinate (A2) at (-1.5, 0);  \filldraw[green] (A2)circle(0.08);
\coordinate (A3) at (0, -1.5); \filldraw[green] (A3) circle(0.08);
\coordinate (A4) at (1.5,-0.3);\filldraw[green] (A4)circle(0.08);
\coordinate (A5) at (1.2,1.0);\filldraw[green] (A5)circle(0.08);
\draw node[above] at (A1) {$u_{b1}$}; \draw node[left] at (A2){$u_{b2}$};
\draw node[below] at (A3){$u_{b3}$}; \draw node[right] at (A4) {$u_{b4}$}; \draw node[right] at (A5) {$u_{b5}$};
\coordinate (M1) at (-0.85,0.75); \coordinate (M2) at (-0.75,-0.75); \coordinate (M3) at (0.75,-0.9);
\coordinate (M4) at (1.35,0.35);  \coordinate (M5) at (0.5,1.25);  \coordinate (MM) at (0.1,0.15);
\coordinate (M1end) at (-1.55,1.35); \coordinate (M2end) at (-1.35,-1.4); \coordinate (M3end) at (1.32,-1.6);
\coordinate (M4end) at (2.25,0.55);   \coordinate (M5end) at (0.82,2.1); \coordinate (center) at (0.1,0.15);
\coordinate (ne1) at (-1.15,1.00);\coordinate (ne2) at (-1,-1);\coordinate (ne3) at (1.05,-1.2);
\coordinate (ne4) at (2.05,0.50);\coordinate (ne5) at (1.0,1.75);
\draw node[left] at (ne1) {$u_{n1}$};\draw node[left] at (ne2) {$u_{n2}$};\draw node[right] at (ne3) {$u_{n3}$};
\draw node[below] at (ne4) {$u_{n4}$};\draw node[below] at (ne5) {$u_{n5}$}; \draw node at (center) {T};
\draw node at (center) {T}; \draw (A1)--(A2)--(A3)--(A4)--(A5)--cycle;
\filldraw[red] (M1) circle(0.08);\filldraw[red] (M2) circle(0.08);\filldraw[red] (M3) circle(0.08);
\filldraw[red] (M4) circle(0.08);\filldraw[red] (M5) circle(0.08);
\draw[->] (M1)--(M1end);\draw[->] (M2)--(M2end);\draw[->] (M3)--(M3end);\draw[->] (M4)--(M4end);\draw[->] (M5)--(M5end);
\end{tikzpicture}
\caption{Local degrees of freedom on a triangular element (left) and a pentagonal element   (right).}\label{polygonal-element}
\end{center}
\end{figure}
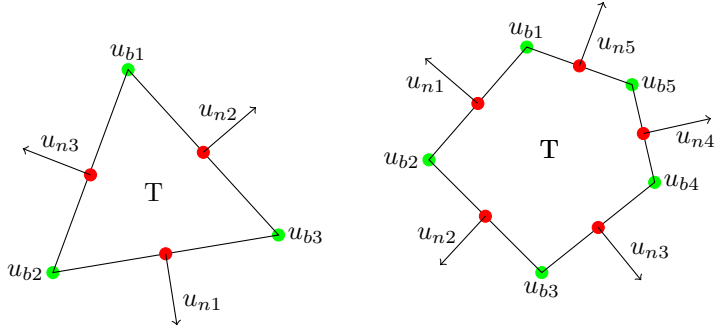
\end{remark}

\section{Solution existence and uniqueness}\label{Section:seu}

On each element $T\in{\cal T}_h$, let $Q_0$ be the usual $L^2$ projection operator onto $P_2(T)$. Then for any $\phi\in H^2(\O)$, we define a projection $Q_h\phi\in V_h$ such that on each element $T,$
$$
Q_h\phi=\{Q_0\phi,Q_b\phi, Q_n( \nabla\phi\cdot\bn_{f})\bn_{f}\}.
$$
Moreover, let ${\S}$ be the locally defined $L^2$ projection operator onto the space $P_0(T)$.

\begin{lemma}  \label{commutative properties}
The aforementioned projection operators $Q_{h}$, $Q_n$ and ${\S}$ satisfy the following commutative properties:
\begin{eqnarray}
\label{comu2}
\nabla_{w,\F}Q_{h}\phi&=&Q_n(\bn_{f}\times(\nabla\phi\times\bn_{f})),\\
\label{comu1}
\pa_{ij,w}^2(Q_h\phi)&=&{\S}\pa_{ij}^2\phi,\quad  i,j=1,\ldots,d
\end{eqnarray}
for any $\phi\in H^2(T).$
\end{lemma}

  \begin{proof}
The proof of the identity \eqref{comu2} can be found in \cite{WWYZ-JCAM2021}. To derive the other identity \eqref{comu1}, from \eqref{vgi} and \eqref{comu2}, we arrive at
\begin{eqnarray}\label{decomposition-new}
 \begin{split}
(Q_h\phi)_g&=Q_n(\nabla\phi\cdot\bn_{f})\bn_{f}+Q_n(\bn_{f}\times(\nabla\phi\times\bn_{f}))\\
&=Q_n(\nabla\phi).
 \end{split}
\end{eqnarray}
Next, using \eqref{discrete weak partial derivetive}, \eqref{decomposition-new}, the definitions of $Q_n$, ${\S}$ and the usual integration by parts gives
\begin{equation*}
\begin{split}
 &(\pa_{ij,w}^2(Q_h\phi),\varphi)_T\\
=&\langle(Q_n(\nabla\phi))_i,\varphi n_j\rangle_{\pa T}\\
 =&\langle (\nabla\phi)_i,\varphi n_j\rangle_{\pa T}\\
=&\langle\partial_i\phi,\varphi n_j\rangle_{\pa T}+(\phi,\pa^2_{ji}\varphi)_T-\langle\phi n_i,\pa_j\varphi\rangle_{\pa T}\\
=&(\pa^2_{ij}\phi,\varphi)_T\\
=&({\S}\pa_{ij}^2\phi,\varphi)_T
\end{split}
\end{equation*}
for any $\varphi\in P_0(T)$, which implies \eqref{comu1} holds true. This completes the proof.
 \end{proof}

For any $v\in V_h$, we define a semi-norm induced by the WG scheme \eqref{WG-scheme}; i.e.,
\begin{equation}\label{tri-semibar}
\3barv\3bar=\Big(a(v,v)\Big)^{1/2}.
\end{equation}

\begin{lemma}\label{tri-norm}
For any $v\in V_h^0$, the semi-norm $\3barv\3bar$ given by \eqref{tri-semibar} defines a norm.
\end{lemma}
\begin{proof}
It suffices to verify the positive property for $\3barv\3bar$. To this end, we assume that $\3barv\3bar=0$ for some $v\in V_h^0$. It follows from \eqref{tri-semibar} that $\pa^2_wv=0$ and $s(v,v)=0$, which indicates $\pa^2_{ij,w}v=0$ for any $i,j=1,\ldots,d$ on each $T\in{\cal T}_h$, $Q_bv_0=v_b$ on each  $\pa\F$ and $ Q_n(\nabla v_0)\cdot\bn_{f}=v_n$ on each $\pa T$. This further leads to $Q_hv_0=v$ on each $T$. Using the definition of ${\S}$ and \eqref{comu1} gives
$$
\partial_{ij}^2 v_0 = {\S}\partial_{ij}^2 v_0 = \partial_{ij,w}^2 (Q_hv_0) =\partial_{ij,w}^2 v =0,\quad i,j=1,\ldots, d.
$$
Thus, we have $\nabla v_0=const$ on each $T$. It follows from \eqref{decomposition-new} with $\phi=v_0$, \eqref{vgi}, $Q_n(\nabla v_0)\cdot\bn_{f}=v_n$ on each $\pa T$ and $Q_hv_0=v$ on each $T$ that $Q_n(\bn_{f}\times(\nabla v_0\times\bn_{f}))=\nabla_{w,\F}v$ on each $\pa T$, which implies $\nabla v_0=v_n \bn_{f}+\nabla_{w,\F}v$ on each face $\F\in\F_h$ and hence $\nabla v_0\in C^0(\O)$. Using $v_b=0$ on $e\subset\pa\O$ and \eqref{weak tangential derivative}, we have $\nabla_{w,\F}v=\textbf{0}$ on each $\F\subset\pa\O$. This, along with 
$v_n=0$ on each $\F\subset\pa\O$, gives rise to $\nabla v_0=\textbf{0}$ on each $\F\subset\pa\O$. From $\nabla v_0\in C^0(\O)$ and $\nabla v_0=\textbf{0}$ on $\F\subset\pa\O$, we have $\nabla v_0=\textbf{0}$ in $\O$ and further $v_0=const$ on each $T$. This yields $v_n=0$ on each $\pa T$ due to $v_n=Q_n(\nabla v_0)\cdot\bn_{f}$ on each $\pa T$. Furthermore, it follows from $Q_bv_0=v_b$ and $Q_bv_0=v_0$ on each $\pa\F$ that $v_0\in C^1(\Omega)$ and thus $v_0=const$ in $\Omega$. Using $v_b=0$ on $\pa\O$, we have $v_0=0$ in $\O$ and hence $v_b=0$ on each $\pa \F$. This completes the proof of the lemma.
\end{proof}

\begin{lemma}\label{unique}
The WG scheme \eqref{WG-scheme} has a unique numerical solution.
\end{lemma}

\begin{proof}
It suffices to prove that zero is the unique solution of the WG scheme \eqref{WG-scheme} with homogeneous conditions $f=0$, $g=0$ and $\nu=0$. To this end, by setting $v=u_h$ in \eqref{WG-scheme}, we arrive at  
$$a(u_h,u_h)=0,$$
which, together with Lemma \ref{tri-norm}, leads to $u_h\equiv0$. This completes the proof of the lemma.
\end{proof}

\section{Error equations}\label{Section:error-equation}

Denote by $u$ and $u_h\in V_h$ the solutions of the model problem \eqref{model-problem} and the WG scheme \eqref{WG-scheme}, respectively. We define the corresponding error as follows
\begin{equation}\label{error-function}
e_h=Q_hu-u_h.
\end{equation}

\begin{lemma}\label{error equation}
Let $e_h$ be the error function given by \eqref{error-function}. Then, the following equation holds true:
\begin{eqnarray}\label{Error-equation}
a(e_h,v)=\zeta_u(v),\qquad\forall v\in V_h^0,
\end{eqnarray}
where $\zeta_u(v)$ is defined by
 \begin{equation}\label{error equation-remainder}
\begin{split}
\zeta_u(v)
=&s(Q_hu,v)-\sum_{T\in{\cal T}_h}\sum_{i,j=1}^d\langle v_0,\pa_j(\pa^2_{ij}u)n_i\rangle_{\pa T}\\
&+\sum_{T\in{\cal T}_h}\sum_{i,j=1}^d\langle(\pa_iv_0-v_{gi})n_j,(I-{\S})\pa^2_{ij}u\rangle_{\pa T}.
\end{split}
\end{equation}
\end{lemma}

\begin{proof}
For each face $\F\subset\pa\O$, using \eqref{weak tangential derivative} and $v_b=0$ on $e\subset\pa\O$, we have
$$
\langle\nabla_{w,\F}v,\pmb{\psi}\times\bn_{f}\rangle_{{\F}}
=\langle v_b,\pmb{\psi}\cdot\pmb{\tau}\rangle_{\pa\F}=0,\quad\forall\pmb{\psi}\in{{\W}}_0(\F),
 $$
which implies $\nabla_{w,\F}v=\textbf{0}$ on $\pa \Omega$ and hence $\pmb{v_g}=\textbf{0}$ on $\pa\O$ due to $v_n=0$ on each $\F\subset\pa\O$ and \eqref{vgi}.
 
Testing the model equation \eqref{model-problem} by $v_0$ and using the usual integration by parts give
\begin{equation}\label{error equation-1}
\begin{split}
  &(f,v_0)\\
   =&\sum_{T\in{\cal T}_h}(\Delta^2u,v_0)_T\\
   =&\sum_{T\in{\cal T}_h}\sum_{i,j=1}^d (\pa^2_{ij}u,\pa^2_{ij}v_0)_T-\langle\pa^2_{ij}u,\pa_iv_0 n_j\rangle_{\pa T}
  +\langle\pa_j(\pa^2_{ij}u)n_i,v_0\rangle_{\pa T} \\
   =&\sum_{T\in{\cal T}_h}\sum_{i,j=1}^d (\pa^2_{ij}u,\pa^2_{ij}v_0)_T-\langle\pa^2_{ij}u,(\pa_iv_0- v_{gi})n_j\rangle_{\pa T}
 +\langle\pa_j(\pa^2_{ij}u)n_i,v_0\rangle_{\pa T},
\end{split}
\end{equation}
where on the last line we used the fact $\sum_{T\in{\cal T}_h}\sum_{i,j=1}^d\langle\pa^2_{ij}u,v_{gi}n_j\rangle_{\pa T}=0$ due to $\pmb{v_g}=\textbf{0}$ on $\pa\O.$

Next, we deal with the first term on the right hand of \eqref{error equation-1}. By taking $\varphi={\S}\pa^2_{ij}u\in P_0(T)$ in \eqref{error equation-13} and using \eqref{comu1}, we arrive at
\begin{equation}\label{error equation-3}
\begin{split}
(\pa_{ij}^2u,\pa_{ij}^2v_0)_T&=({\S}\pa_{ij}^2u,\pa_{ij}^2v_0)_T\\
&=(\pa_{ij,w}^2v,{\S}\pa^2_{ij}u)_T+\langle(\pa_iv_0-v_{gi})n_j,{\S}\pa^2_{ij}u\rangle_{\pa T}\\
&=(\pa_{ij,w}^2v, \pa_{ij,w}^2(Q_hu))_T+\langle(\pa_iv_0-v_{gi})n_j,{\S}\pa^2_{ij}u\rangle_{\pa T}.
\end{split}
\end{equation}
Combining \eqref{error equation-3} with \eqref{error equation-1} yields
\begin{equation}\label{error equation-remainder-4}
\begin{split}
  \sum_{T\in{\cal T}_h}\sum_{i,j=1}^d(\pa_{ij,w}^2v,\pa_{ij,w}^2Q_hu)_T
=&(f,v_0)-\sum_{T\in{\cal T}_h}\sum_{i,j=1}^d\langle v_0,\pa_j(\pa^2_{ij}u)n_i\rangle_{\pa T}\\
&+\sum_{T\in{\cal T}_h}\sum_{i,j=1}^d\langle(\pa_iv_0-v_{gi})n_j,(I-{\S})\pa^2_{ij}u\rangle_{\pa T}.
\end{split}
\end{equation}
Finally, the difference of the WG scheme \eqref{WG-scheme} and \eqref{error equation-remainder-4} gives rise to \eqref{Error-equation}. This completes the proof.
\end{proof}

\section{Technical inequalities}\label{technique-estimate}
For any $T\in{\cal T}_h$ and $\phi\in H^1(T)$, we have the following trace inequality \cite{WY-ellip_MC2014}:
\begin{equation}\label{trace}
\|\phi\|_{\pa T}^2\lesssim h_T^{-1}\|\phi\|_T^2+h_T\|\nabla\phi\|_T^2.
\end{equation}
Moreover, if $\phi$ is a polynomial on $T\in {\cal T}_h$, using the standard inverse inequality, there holds
\begin{equation}\label{trace1}
\|\phi\|_{\pa T}^2\lesssim h_T^{-1}\|\phi\|_T^2.
\end{equation}

\begin{lemma}\cite{WY-ellip_MC2014}\label{error projection}
 Assume the finite element partition ${\cal T}_h$ is shape regular as specified in \cite{WY-ellip_MC2014}.  Let $0\leq s\leq2$. Then, for any $\phi\in H^3(\O)$, we have
\begin{eqnarray}
&\sum_{T\in{\cal T}_h}h_T^{2s}\|\phi-Q_0\phi\|_{s,T}^2\lesssim h^{6}\|\phi\|_{3}^2,\label{estimate-1}\\
&\sum_{T\in{\cal T}_h}\sum_{i,j=1}^dh_T^{2s}\|\pa^2_{ij}\phi-{\S}\pa^2_{ij}\phi\|_{s,T}^2\lesssim h^{2}\|\phi\|_{3}^2.\label{estimate-2}
\end{eqnarray}
\end{lemma}

On each element $T\in{\cal T}_h$, let $\F\subset\pa T$ be a face consisting of edges $e_m$ for $m=1,\ldots, M$. We introduce a linear operator ${\mathcal S}$ mapping $v_b$ from a piecewise constant function  to a piecewise linear function on $\F$ through the least-squares approach to minimize
\begin{equation}\label{linear-operator}
 \sum_{m=1}^M|{\mathcal S}(v_b)(A_m)-v_b(A_m)|^2,
\end{equation}
where $\{A_m\}_{m=1}^{M}$ are the two end points of $\F$ when $d=2$, and $\{A_m\}_{m=1}^{M}$ is the set of midpoints of $e_m$ for $m=1,\ldots,M$ when $d=3$. Denote by $|e_m|$ the length of edge $e_m$.

 \begin{lemma}\label{PROJECTION-ESTIMATE-66}
For any $v\in V_h$, there holds
\begin{equation}\label{21:19}
\sum_{\F\in\F_h}h_T^{-1}\|{\mathcal S}(v_b)\|_{\F}^2\lesssim\sum_{\F\in\F_h}\|v_b\|_{\pa\F}^2.
\end{equation}
\end{lemma}
\begin{proof}
It follows from  \eqref{linear-operator} that
\begin{equation*}
\begin{split}
\sum_{\F\in\F_h}h_T^{-1}\|{\mathcal S}(v_b)\|_{\F}^2
\lesssim&\sum_{\F\in\F_h}\sum_{m=1}^Mh_T^{-1}|{\mathcal S}(v_b)(A_m)|^2h_T^{d-1}\\
\lesssim&\sum_{\F\in\F_h}\sum_{m=1}^Mh_T^{d-2}(|{\mathcal S}(v_b)(A_m)-v_b(A_m)|^2+|v_b(A_m)|^2)\\
\lesssim&\sum_{\F\in\F_h}\sum_{m=1}^Mh_T^{d-2}(|v_b(A_m)|^2+|v_b(A_m)|^2)\\
\lesssim&\sum_{\F\in\F_h}\|v_b\|_{\pa\F}^2.
\end{split}
\end{equation*}
This completes the proof of the lemma.
\end{proof}

 \begin{lemma}\label{PROJECTION-ESTIMATE-68}
For any $v\in V_h$, we have the following estimate
\begin{equation}\label{21:17}
\Big(\sum_{T\in{\cal T}_h}\sum_{i=1}^dh_T^{-1}\|Q_n(\pa_iv_0)-v_{gi}\|_{\pa T}^2\Big)^{\frac{1}{2}}\lesssim
\3barv\3bar.
\end{equation}
\end{lemma}

\begin{proof}
By using $\nabla v_0=(\nabla v_0\cdot\bn_{f})\bn_{f}+\bn_{f}\times(\nabla v_0\times\bn_{f})$, \eqref{vgi} and \eqref{tri-semibar} gives
\begin{equation}\label{0:33}
\begin{split}
&\Big(\sum_{T\in{\cal T}_h}\sum_{i=1}^dh_T^{-1}\|Q_n(\pa_iv_0)-v_{gi}\|_{\pa T}^2\Big)^{\frac{1}{2}}\\
=&\Big(\sum_{T\in{\cal T}_h}h_T^{-1}\|Q_n(\nabla v_0)-\pmb{v_g}\|_{\pa T}^2\Big)^{\frac{1}{2}}\\
=&\Big(\sum_{T\in{\cal T}_h}h_T^{-1}\|Q_n(\nabla v_0\cdot\bn_{f})\bn_{f}+Q_n(\bn_{f}\times(\nabla v_0\times\bn_{f}))
\\&-(v_n\bn_{f}+\nabla_{w,\F}v)\|_{\pa T}^2\Big)^{\frac{1}{2}}\\
\lesssim&\Big(\sum_{T\in{\cal T}_h}h_T^{-1}\|Q_n(\nabla v_0\cdot\bn_{f})-v_n\|_{\pa T}^2\\&+\sum_{\F\in\F_h}h_T^{-1}\|Q_n(\bn_{f}\times(\nabla v_0\times\bn_{f}))-\nabla_{w,\F}v\|_{\F}^2\Big)^{\frac{1}{2}}\\
\lesssim&\Big(\3barv\3bar^2+\sum_{\F\in\F_h}h_T^{-1}\|Q_n(\bn_{f}\times(\nabla v_0\times\bn_{f}))-\nabla_{w,\F}v\|_{\F}^2\Big)^{\frac{1}{2}}.
\end{split}
\end{equation}
 To estimate the second term on the last line in \eqref{0:33}, it follows from \eqref{weak tangential derivative}, the Stokes theorem and  \eqref{trace1} that
\begin{equation*}
\begin{split}
|\langle Q_n(\bn_{f}\times(\nabla v_0\times\bn_{f}))-\nabla_{w,\F}v,\pmb{\psi}\times\bn_{f}\rangle_{\F}|
=&|\langle Q_bv_0-v_b,\pmb{\psi}\cdot\pmb{\tau}\rangle_{\pa\F}|\\
\lesssim&\|Q_bv_0-v_b\|_{\pa\F}\|\pmb{\psi}\|_{\pa\F}\\
\lesssim&h_T^{-\frac{1}{2}}\|Q_bv_0-v_b\|_{\pa\F}\|\pmb{\psi}\|_{\F}
\end{split}
\end{equation*}
for any $\pmb{\psi}\in{{\W}}_0(\F)$, and thus
\begin{equation}\label{0:36}
\begin{split}
\|Q_n(\bn_{f}\times(\nabla v_0\times\bn_{f}))-\nabla_{w,\F}v\|_{\F}\lesssim h_T^{-\frac{1}{2}}\|Q_bv_0-v_b\|_{\pa\F}.
\end{split}
\end{equation}
Combining \eqref{0:36} with \eqref{0:33} and \eqref{tri-semibar} leads to the desired estimate \eqref{21:17}.
\end{proof}

\begin{lemma}\label{PROJECTION-ESTIMATE-6}
For any $v\in V_h$, the following estimate holds true:
\begin{equation}\label{21:16}
\sum_{T\in{\cal T}_h}|v_0|_{2,T}^2\lesssim\3barv\3bar^2.
\end{equation}
\end{lemma}

\begin{proof}
It follows from \eqref{error equation-13} with $\varphi=\pa_{ij}^2v_0$ that
\begin{equation*}
\begin{split}
(\pa_{ij}^2v_0,\pa_{ij}^2v_0)_T=&(\pa_{ij,w}^2v,\pa_{ij}^2v_0)_T+\langle(\pa_iv_0-v_{gi})n_j,\pa_{ij}^2v_0\rangle_{\pa T}\\
=&(\pa_{ij,w}^2v,\pa_{ij}^2v_0)_T+\langle( Q_n(\pa_iv_0)-v_{gi})n_j,\pa_{ij}^2v_0\rangle_{\pa T}.
\end{split}
\end{equation*}
Using the Cauchy-Schwarz inequality, \eqref{21:17} and \eqref{trace1}, we have 
\begin{equation*}
\begin{split}
\sum_{T\in{\cal T}_h}|v_0|_{2,T}^2
\lesssim&\Big(\sum_{T\in{\cal T}_h}\sum_{i,j=1}^d\|\pa_{ij,w}^2v\|_T^2\Big)^{\frac{1}{2}}
   \Big(\sum_{T\in{\cal T}_h}\sum_{i,j=1}^d\|\pa_{ij}^2v_0\|_T^2\Big)^{\frac{1}{2}}\\
&+\Big(\sum_{T\in{\cal T}_h}\sum_{i=1}^dh_T^{-1}\|Q_n(\pa_iv_0)-v_{gi}\|_{\pa T}^2\Big)^{\frac{1}{2}}
   \Big(\sum_{T\in{\cal T}_h}\sum_{i,j=1}^dh_T\|\pa_{ij}^2v_0\|_{\pa T}^2\Big)^{\frac{1}{2}}\\
\lesssim&  \3barv\3bar\Big(\sum_{T\in{\cal T}_h}|v_0|_{2,T}^2\Big)^{\frac{1}{2}},
\end{split}
\end{equation*}
which gives rise to \eqref{21:16}. This completes the proof of the lemma.
\end{proof}

\section{Error estimates}\label{Section:EE}
We start this section by establishing the error estimate for the numerical approximation in the energy norm.

\begin{theorem}\label{THM:energy-estimate}
Let $u$ and $u_h\in V_h$ be the exact solution of the model problem \eqref{model-problem} and the numerical approximation of the WG scheme \eqref{WG-scheme}, respectively. Assume that the exact solution $u$ satisfies $u\in H^4(\O)$. Then, we have the following error estimate:
\begin{equation}\label{energy-estimate-0}
\3bare_h\3bar\lesssim h\|u\|_4.
\end{equation}
\end{theorem}

\begin{proof}
It follows from \eqref{Error-equation} with $v=e_h\in V_h^0$ that
\begin{equation}\label{energy-estimate-2}
\begin{split}
\3bare_h\3bar^2
=&s(Q_hu,e_h)-\sum_{T\in{\cal T}_h}\sum_{i,j=1}^d\langle e_0,\pa_j(\pa^2_{ij}u)n_i\rangle_{\pa T}\\
&+\sum_{T\in{\cal T}_h}\sum_{i,j=1}^d\langle(\pa_ie_0-e_{gi})n_j,(I-{\S})\pa^2_{ij}u\rangle_{\pa T}\\
=&J_1+J_2+J_3.
\end{split}
\end{equation}

To analyze $J_1$, by using the Cauchy-Schwarz inequality, \eqref{trace} and \eqref{estimate-1} yields
\begin{equation}\label{energy-estimate-3}
\begin{split}
|J_1|
\leq&|\sum_{T\in{\cal T}_h}h_T^{-2}\langle Q_b(Q_0u)-Q_bu,Q_be_0-e_b\rangle_{\pa\F}|\\
&+|\sum_{T\in{\cal T}_h}h_T^{-1}\langle Q_n(\nabla Q_0u)\cdot\bn_{f}- Q_n(\nabla u\cdot\bn_{f}),
     Q_n(\nabla e_0)\cdot\bn_{f}-e_n\rangle_{\pa T}|\\
\lesssim&\Big(\sum_{T\in{\cal T}_h}h_T^{-2}\|Q_0u-u\|_{\pa\F}^2\Big)^{\frac{1}{2}}
   \Big(\sum_{T\in{\cal T}_h}h_T^{-2}\|Q_be_0-e_b\|_{\pa\F}^2\Big)^{\frac{1}{2}}\\
&+\Big(\sum_{T\in{\cal T}_h}h_T^{-1}\|\nabla Q_0u-\nabla u\|_{\pa T}^2\Big)^{\frac{1}{2}}
    \Big(\sum_{T\in {\cal T}_h}h_T^{-1}\| Q_n(\nabla e_0)\cdot\bn_{f}-e_n\|_{\pa T}^2\Big)^{\frac{1}{2}}\\
\lesssim&\Big(\sum_{T\in{\cal T}_h}h_T^{-4}\|Q_0u-u\|_{T}^2+h_T^{-2}\|\nabla(Q_0u-u)\|_{T}^2+|Q_0u-u|_{2,T}^2\Big)^{\frac{1}{2}}\3bare_h\3bar\\
&+\Big(\sum_{T\in{\cal T}_h}h_T^{-2}\|\nabla Q_0u-\nabla u\|_{T}^2+|Q_0u-u|_{2,T}^2\Big)^{\frac{1}{2}}\3bare_h\3bar \\
\lesssim& h\|u\|_3\3bare_h\3bar.
\end{split}
\end{equation}

To deal with $J_2$, using \eqref{linear-operator}, $e_b=0$ on $\pa\O$, the Cauchy-Schwarz inequality, \eqref{21:19} and \eqref{trace}, we have
\begin{equation*}
\begin{split}
|J_2|=&|\sum_{T\in{\cal T}_h}\sum_{i,j=1}^d\langle e_0,\pa_j(\pa^2_{ij}u)n_i\rangle_{\pa T}|\\
=&|\sum_{T\in{\cal T}_h}\sum_{i,j=1}^d\langle e_0-{\mathcal S}(e_b),\pa_j(\pa^2_{ij}u)n_i\rangle_{\pa T}|\\
=&|\sum_{T\in{\cal T}_h}\sum_{i,j=1}^d\langle(e_0-{\mathcal S}(e_0))+({\mathcal S}(e_0)-{\mathcal S}(Q_be_0))+{\mathcal S}(Q_be_0-e_b),\pa_j(\pa^2_{ij}u)n_i\rangle_{\pa T}|\\
\lesssim&\Big(\sum_{\F\in{\F}_h}h_T^{-3}\|e_0-{\mathcal S}(e_0)\|_{\F}^2+h_T^{-3}\|{\mathcal S}(e_0)-{\mathcal S}(Q_be_0)\|_{\F}^2
+h_T^{-3}\|{\mathcal S}(Q_be_0-e_b)\|_{\F}^2\Big)^{\frac{1}{2}}\\
&\cdot\Big(\sum_{T\in{\cal T}_h}\sum_{i,j=1}^dh_T^3\|\pa_j(\pa^2_{ij}u)\|_{\pa T}^2\Big)^{\frac{1}{2}}\\
\lesssim&\Big(\sum_{\F\in{\F}_h}h_T\|D_{\pmb{\tau\tau}}e_0\|_{\F}^2+\sum_{m=1}^Mh_T^{-2}\|e_0(A_m)-Q_be_0\|_{e_m}^2
+h_T^{-2}\|Q_be_0-e_b\|_{\pa\F}^2\Big)^{\frac{1}{2}}\\
&\cdot\Big(\sum_{T\in{\cal T}_h}h_T^4|u|_{3,T}^2+h_T^{2}|u|_{4,T}^2\Big)^{\frac{1}{2}},
\end{split}
\end{equation*}
where $D_{\pmb{\tau\tau}}e_0$ represents the second tangential derivative on $\F.$\\
For the case of two dimensions, from the definition of $Q_b$, we arrive at
$$\sum_{\F\in{\F}_h}\sum_{m=1}^Mh_T^{-2}\|e_0(A_m)-Q_be_0\|_{e_m}^2=0.$$
For the case of three dimensions, it is clear that
\begin{equation*}
\begin{split}
\sum_{\F\in{\F}_h}\sum_{m=1}^Mh_T^{-2}\|e_0(A_m)-Q_be_0\|_{e_m}^2
\lesssim&\sum_{\F\in{\F}_h}\sum_{m=1}^Mh_T^{-2}\|e_0(A_m)-\frac{1}{2}(e_0(A_{sm})+e_0(A_{dm})) \|_{e_m}^2\\
&+h_T^{-2}|e_j|^4\|\widehat{D}_{\pmb{\tau\tau}}e_0\|_{e_m}^2.
\end{split}
\end{equation*}
We apply \eqref{trace1} and \eqref{21:16} to obtain
\begin{equation}\label{energy-estimate-5}
\begin{split}
|J_2|
\lesssim&\Big(\sum_{T\in{\cal T}_h}|e_0|_{2,T}^2+\sum_{\F\in{\F}_h}\sum_{m=1}^Mh_T^{-2}|e_j|^4\|\widehat{D}_{\pmb{\tau\tau}}e_0\|_{e_m}^2+\3bare_h\3bar^2\Big)^{\frac{1}{2}} h\|u\|_4\\
\lesssim&\Big(\3bare_h\3bar^2+\sum_{T\in{\cal T}_h}|e_0|_{2,T}^2\Big)^{\frac{1}{2}}h\|u\|_4\\
\lesssim&h\|u\|_4\3bare_h\3bar,
\end{split}
\end{equation}
where $\widehat{D}_{\pmb{\tau\tau}}e_0$ denotes the second tangential derivative on $e_m$, $A_{sm}$ and $A_{dm}$ stand for the starting point and ending point of edge $e_m$, respectively.

To estimate the term $J_3$, by using the Cauchy-Schwarz inequality, \eqref{21:17}, \eqref{trace}, \eqref{trace1} and  \eqref{21:16} gives
  \begin{equation}\label{energy-estimate-4}
\begin{split}
|J_3|=&|\sum_{T\in{\cal T}_h}\sum_{i,j=1}^d\langle(\pa_ie_0-e_{gi})n_j,(I-{\S})\pa^2_{ij}u\rangle_{\pa T}|\\
\lesssim&\Big(\sum_{T\in{\cal T}_h}\sum_{i=1}^dh_T^{-1}\|\pa_ie_0-Q_n(\pa_ie_0)\|_{\pa T}^2+h_T^{-1}\|Q_n(\pa_ie_0)-e_{gi}\|_{\pa T}^2\Big)^{\frac{1}{2}}\\
&\cdot\Big(\sum_{T\in{\cal T}_h}\sum_{i,j=1}^dh_T\|(I-{\S})\pa^2_{ij}u\|_{\pa T}^2\Big)^{\frac{1}{2}}\\
\lesssim&\Big(\sum_{\F\in{\F}_h}h_T^{-1}\|\nabla e_0-Q_n(\nabla e_0)\|_{\F}^2+\3bare_h\3bar^2\Big)^{\frac{1}{2}}\\
&\cdot\Big(\sum_{\F\in{\F}_h}\sum_{i,j=1}^dh_T^3\|D_{\pmb{\tau}}(\pa^2_{ij}u)\|_{\F}^2\Big)^{\frac{1}{2}}\\
\lesssim&\Big(\sum_{\F\in{\F}_h}h_T\|D_{\pmb{\tau}}(\nabla e_0)\|_{\F}^2+\3bare_h\3bar^2\Big)^{\frac{1}{2}}h\|u\|_4\\
\lesssim&\Big(\sum_{T\in{\cal T}_h}|e_0|_{2,T}^2+\3bare_h\3bar^2\Big)^{\frac{1}{2}}h\|u\|_4\\
\lesssim&h\|u\|_4\3bare_h\3bar,
\end{split}
\end{equation}
where $D_{\pmb{\tau}}(\pa^2_{ij}u)$ and $D_{\pmb{\tau}}(\nabla e_0)$ represent the tangential components of $\pa^2_{ij}u$ and $\nabla e_0$ on $\F$, respectively.\\
Finally, combining \eqref{energy-estimate-3}-\eqref{energy-estimate-4} with \eqref{energy-estimate-2} leads to \eqref{energy-estimate-0}. This completes the proof of the theorem.
\end{proof}

We shall establish the error estimate for the numerical approximation in the usual $L^2$ norm. To this end, let us consider the following dual problem
\begin{equation}\label{dual-equation}
\begin{split}
\Delta^2\Phi&=e_0,\quad\mbox{in}~~\O,\\
        \Phi&=0,~\quad\mbox{on}~~\pa\O,\\
\frac{\pa\Phi}{\pa\textbf{n}}&=0,~\quad\mbox{on}~~\pa\O.
\end{split}
\end{equation}
Assume that the dual problem \eqref{dual-equation} has the $H^4$-regularity property in the sense that the solution $\Phi$ satisfies $\Phi\in H^4(\O)$ and the following priori estimate:
\begin{equation}\label{dual-regular}
\|\Phi\|_4\leq C\|e_0\|.
\end{equation}

\begin{theorem}\label{THM:L2-estimate-e0}
Let $u$ and $u_h\in V_h$ be the exact solution of the model problem \eqref{model-problem} and the numerical approximation of the WG scheme \eqref{WG-scheme}, respectively. Assume that the exact solution $u\in H^4(\O)$ and the $H^4$-regularity property \eqref{dual-regular} hold true. Then, the following error estimate holds true:
\begin{equation}\label{THM:L2-estimate-0}
\|e_0\|\lesssim h^2\|u\|_4.
\end{equation}
\end{theorem}

\begin{proof}
On each face $\F\subset\pa\O$, it follows from \eqref{weak tangential derivative} and $e_b=0$ on $e\subset\pa\O$ that $\nabla_{w,\F}e_h=\textbf{0}$ on $\pa \Omega$, which together with $e_n=0$ on each $\F\subset\pa\O$ and \eqref{vgi}, leads to $\pmb{e_g}=\textbf{0}$ on $\pa\O.$

By testing the dual problem \eqref{dual-equation} by $e_0$ and using the usual integration by parts gives
\begin{equation}\label{THM:L2-estimate-1}
\begin{split}
 \|e_0\|^2
=&(\Delta^2\Phi,e_0)\\
=&\sum_{T\in{\cal T}_h}\sum_{i,j=1}^d (\pa^2_{ij}\Phi,\pa^2_{ij}e_0)_T-\langle\pa^2_{ij}\Phi,\pa_ie_0n_j\rangle_{\pa T}
   +\langle\pa_j(\pa^2_{ij}\Phi)n_i,e_0\rangle_{\pa T} \\
=&\sum_{T\in{\cal T}_h}\sum_{i,j=1}^d (\pa^2_{ij}\Phi,\pa^2_{ij}e_0)_T-\langle\pa^2_{ij}\Phi,(\pa_ie_0-e_{gi})n_j\rangle_{\pa T}\\
 &+\langle\pa_j(\pa^2_{ij}\Phi)n_i,e_0\rangle_{\pa T},
\end{split}
\end{equation}
where we also have used the fact $\sum_{T\in{\cal T}_h}\sum_{i,j=1}^d\langle\pa^2_{ij}\Phi,e_{gi}n_j\rangle_{\pa T}=0$ due to $\pmb{e_g}=\textbf{0}$ on $\pa\O.$

To deal with the first term on the third line in \eqref{THM:L2-estimate-1}, it follows from \eqref{error equation-3} with $u=\Phi$, $v=e_h$ and \eqref{Error-equation} with $v=Q_h\Phi\in V_h^0$ that
\begin{equation}\label{THM:L2-estimate-2}
\begin{split}
 &\sum_{T\in{\cal T}_h}\sum_{i,j=1}^d(\pa_{ij}^2\Phi,\pa_{ij}^2e_0)_T\\
=&\sum_{T\in{\cal T}_h}\sum_{i,j=1}^d(\pa_{ij,w}^2e_h,\pa_{ij,w}^2(Q_h\Phi))_T
   + \langle(\pa_ie_0-e_{gi})n_j,{\S}\pa^2_{ij}\Phi\rangle_{\pa T}\\
=&-s(e_h,Q_h\Phi)+\zeta_u(Q_h\Phi)+\sum_{T\in{\cal T}_h}\sum_{i,j=1}^d\langle(\pa_ie_0-e_{gi})n_j,{\S}\pa^2_{ij}\Phi\rangle_{\pa T}.
\end{split}
\end{equation}
By inserting \eqref{THM:L2-estimate-2} into \eqref{THM:L2-estimate-1} and then combining \eqref{error equation-remainder}, we have
\begin{equation}\label{THM:L2-estimate-3}
\begin{split}
\|e_0\|^2
=&\zeta_u(Q_h\Phi)-\zeta_\Phi(e_h)\\
=&\sum_{i=1}^3I_i-\zeta_\Phi(e_h),
\end{split}
\end{equation}
where $I_i$ for $i=1,2,3$ are given by \eqref{error equation-remainder} with $v=Q_h\Phi.$

Next, it suffices to estimate each of four terms on the second line of \eqref{THM:L2-estimate-3}. As to the term $I_1$, using the Cauchy-Schwarz inequality, \eqref{trace}, \eqref{estimate-1} and \eqref{dual-regular}, we have
\begin{equation}\label{THM:L2-estimate-5}
\begin{split}
&|I_1|\\
=&|\sum_{T\in{\cal T}_h}h_T^{-2}\langle Q_b(Q_0u)-Q_bu,Q_b(Q_0\Phi)-Q_b\Phi\rangle_{\pa\F}\\
&+h_T^{-1}\langle Q_n(\nabla Q_0u)\cdot\bn_{f}- Q_n(\nabla u\cdot\bn_{f}),
         Q_n(\nabla Q_0\Phi)\cdot\bn_{f}- Q_n(\nabla\Phi\cdot\bn_{f})\rangle_{\pa T}|\\
\lesssim&\Big(\sum_{T\in{\cal T}_h}h_T^{-2}\|Q_0u-u\|_{\pa\F}^2\Big)^{\frac{1}{2}}
         \Big(\sum_{T\in{\cal T}_h}h_T^{-2}\|Q_0\Phi-\Phi\|_{\pa\F}^2\Big)^{\frac{1}{2}}\\
&+\Big(\sum_{T\in{\cal T}_h}h_T^{-1}\|\nabla Q_0u-\nabla u\|_{\pa T}^2\Big)^{\frac{1}{2}}
  \Big(\sum_{T\in{\cal T}_h}h_T^{-1}\|\nabla Q_0\Phi-\nabla\Phi\|_{\pa T}^2\Big)^{\frac{1}{2}}\\
\lesssim&\Big(\sum_{T\in{\cal T}_h}h_T^{-4}\|Q_0u-u\|_{T}^2+h_T^{-2}\|\nabla(Q_0u-u)\|_{T}^2+|Q_0u-u|_{2,T}^2\Big)^{\frac{1}{2}}\\
&
        \cdot\Big(\sum_{T\in{\cal T}_h}h_T^{-4}\|Q_0\Phi-\Phi\|_{T}^2+h_T^{-2}\|\nabla(Q_0\Phi-\Phi)\|_{T}^2+|Q_0\Phi-\Phi|_{2,T}^2\Big)^{\frac{1}{2}}\\
        &+h^2\|u\|_3\|\Phi\|_3\\
\lesssim&h^2\|u\|_3\|\Phi\|_3\\
\lesssim&h^2\|u\|_3\|e_0\|.
\end{split}
\end{equation}

To estimate the term $I_2$, from the fact that $\Phi=0$ on $\pa\O$ and $\sum_{T\in{\cal T}_h}\sum_{i,j=1}^d\langle\Phi,\pa_j(\pa^2_{ij}u)n_i\rangle_{\pa T}=0$, we arrive at
\begin{equation*}
\begin{split}
J_2=&-\sum_{T\in{\cal T}_h}\sum_{i,j=1}^d\langle Q_0\Phi,\pa_j(\pa^2_{ij}u)n_i\rangle_{\pa T}\\
=&-\sum_{T\in{\cal T}_h}\sum_{i,j=1}^d\langle(Q_0\Phi-\Phi)+\Phi,\pa_j(\pa^2_{ij}u)n_i\rangle_{\pa T}\\
=&\sum_{T\in{\cal T}_h}\sum_{i,j=1}^d\langle \Phi-Q_0\Phi,\pa_j(\pa^2_{ij}u)n_i\rangle_{\pa T}.
\end{split}
\end{equation*}
By using the Cauchy-Schwarz inequality, \eqref{trace}, \eqref{estimate-1} and \eqref{dual-regular} gives
\begin{equation}\label{THM:L2-estimate-6}
\begin{split}
|I_2|\lesssim&\Big(\sum_{T\in{\cal T}_h}\|\Phi-Q_0\Phi\|_{\pa T}^2\Big)^{\frac{1}{2}}
    \Big(\sum_{T\in{\cal T}_h}\sum_{i,j=1}^d\|\pa_j(\pa^2_{ij}u)\|_{\pa T}^2\Big)^{\frac{1}{2}}\\
\lesssim&h^2\|u\|_4\|\Phi\|_4\\
\lesssim&h^2\|u\|_4\|e_0\|.
\end{split}
\end{equation}

To deal with the term $I_3$, from \eqref{decomposition-new}, the Cauchy-Schwarz inequality, \eqref{trace}, \eqref{estimate-1}, \eqref{estimate-2} and \eqref{dual-regular}, we obtain
\begin{equation}\label{THM:L2-estimate-7}
\begin{split}
|I_3|=&|\sum_{T\in{\cal T}_h}\sum_{i,j=1}^d\langle(\pa_i Q_0\Phi-( Q_n\nabla\Phi)_i)n_j,(I-{\S})\pa^2_{ij}u\rangle_{\pa T}|\\
\lesssim&\Big(\sum_{T\in{\cal T}_h}\|\nabla Q_0\Phi-\nabla\Phi\|_{\pa T}^2\Big)^{\frac{1}{2}}
\Big(\sum_{T\in{\cal T}_h}\sum_{i,j=1}^d\|(I-{\S})\pa^2_{ij}u\|_{\pa T}^2\Big)^{\frac{1}{2}}\\
\lesssim&\Big(\sum_{T\in{\cal T}_h}h_T^{-1}\|\nabla Q_0\Phi-\nabla\Phi\|_{T}^2+h_T|Q_0\Phi-\Phi|_{2,T}^2\Big)^{\frac{1}{2}}\\
&\cdot  \Big(\sum_{T\in{\cal T}_h}\sum_{i,j=1}^dh_T^{-1}\|(I-{\S})\pa^2_{ij}u\|_{T}^2+h_T|(I-{\S})\pa^2_{ij}u|_{1,T}^2\Big)^{\frac{1}{2}}\\
\lesssim&h^2\|u\|_3\|\Phi\|_3\\
\lesssim&h^2\|u\|_3\|e_0\|.
\end{split}
\end{equation}

For the last term $\zeta_\Phi(e_h)$, we apply the similar arguments as in \eqref{energy-estimate-3}-\eqref{energy-estimate-4} with $u=\Phi$, \eqref{energy-estimate-0} and \eqref{dual-regular} to obtain
\begin{equation}\label{THM:L2-estimate-4}
\begin{split}
|\zeta_\Phi(e_h)|\lesssim&h\3bare_h\3bar\|\Phi\|_4\\
                 \lesssim&h^2\|u\|_4\|\Phi\|_4\\
                 \lesssim&h^2\|u\|_4\|e_0\|.
\end{split}
\end{equation}
Finally, combining \eqref{THM:L2-estimate-5}-\eqref{THM:L2-estimate-4} with \eqref{THM:L2-estimate-3} verifies \eqref{THM:L2-estimate-0}. This completes the proof of the theorem.
\end{proof}

To establish the error estimates for the numerical approximations defined on the faces $\F_h$ and edges $\E_h$, we introduce
$$
\|e_b\|_{\E_h}=\Big(\sum_{T\in{\cal T}_h}h_T^2\|e_b\|_{\pa{\F}}^2\Big)^{\frac{1}{2}},\qquad
\|e_n\|_{\F_h}=\Big(\sum_{T\in{\cal T}_h}h_T\|e_n\|_{\pa T}^2\Big)^{\frac{1}{2}}.
$$
\begin{theorem}\label{THM:L2-estimate-edge}
Under the assumptions of Theorem \ref{THM:L2-estimate-e0}, the following error estimates hold true:
\begin{equation*}
\begin{split}
&\|e_b\|_{\E_h}\lesssim h^2\|u\|_4,\\
&\|e_n\|_{\F_h}\lesssim h\|u\|_4.
\end{split}
\end{equation*}
\end{theorem}

\begin{proof}
By using the triangular inequality, \eqref{trace1}, \eqref{energy-estimate-0}  and \eqref{THM:L2-estimate-0}, we obtain
\begin{equation*}\label{L2-estimate-edge-2}
\begin{split}
\|e_b\|_{\E_h}
\lesssim&\Big(\sum_{T\in{\cal T}_h}h_T^2\|Q_be_0\|_{\pa{\F}}^2+h_T^2\|e_b-Q_be_0\|_{\pa{\F}}^2\Big)^{\frac{1}{2}}\\
\lesssim&\Big(\sum_{T\in{\cal T}_h}h_T\|e_0\|_{\pa T}^2+h_T^4\3bare_h\3bar^2\Big)^{\frac{1}{2}}\\
\lesssim&\Big(\sum_{T\in{\cal T}_h}\|e_0\|_{T}^2+h^6\|u\|_4^2\Big)^{\frac{1}{2}}\\
\lesssim&h^2\|u\|_4,
\end{split}
\end{equation*}
which leads to the first estimate for $e_b$. The similar argument can be applied to derive the error estimate for $e_n$. This completes the proof.
\end{proof}

\section{Numerical experiments}\label{Section:NE}

Several numerical experiments will be implemented to verify the convergence theory established in previous sections. In our numerical examples, the randomised quadrilateral partition, the hexagonal partition, and the non-convex octagonal partition are generated by PolyMesher package \cite{polymesher_APP2012}(see Figure \ref{Pictu-polygonal} (a)-(c) for initial partitions) and the next level of the partitions are refined by the Lloyd iteration \cite{polymesher_APP2012} (see Figure \ref{Pictu-polygonal} (d)-(f)). The uniform cubic partition is generated by uniformly refining the initial $2\times2\times2$ cubic partition of domain $\O=(0,1)^3$ into $2^N\times2^N\times2^N$  cubes for $N=2,\ldots,5$.  The uniform triangular partition and the uniform rectangular  partition are obtained similarly.
\begin{figure*}[h]
\centering
\subfigure [Level 1]{\label{fig: a1}
\includegraphics[height=0.25\textwidth,width=0.31\textwidth]{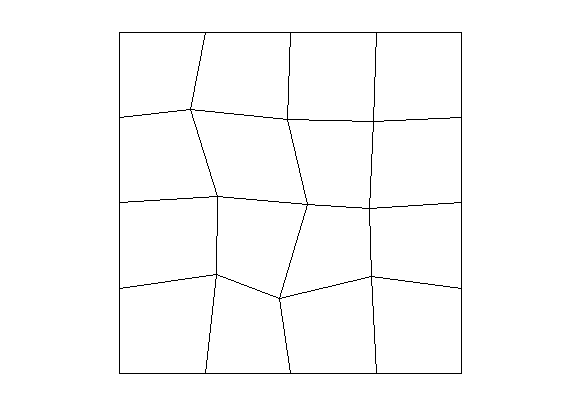}}
\subfigure [Level 1]{\label{fig: b1}
\includegraphics[height=0.25\textwidth,width=0.31\textwidth]{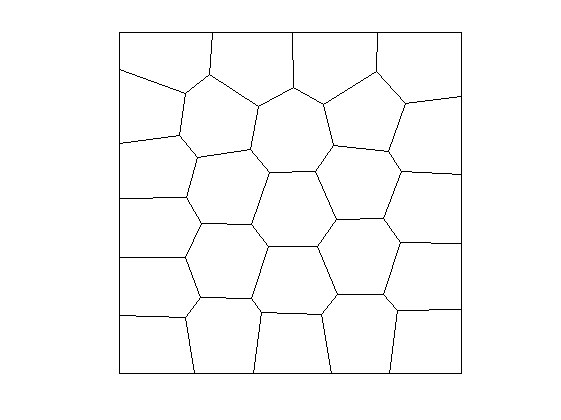}}
\subfigure [Level 1]{\label{fig: c1}
\includegraphics[height=0.25\textwidth,width=0.31\textwidth]{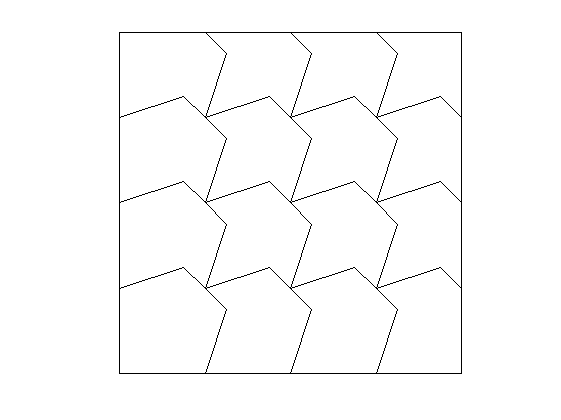}}
\subfigure [Level 2]{\label{fig: a2}
\includegraphics[height=0.25\textwidth,width=0.31\textwidth]{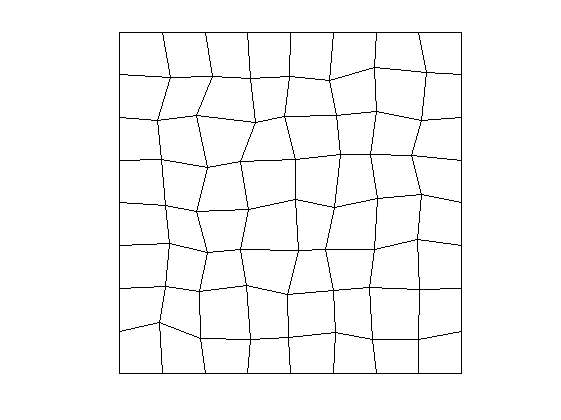}}
\subfigure [Level 2]{\label{fig: b2}
\includegraphics[height=0.25\textwidth,width=0.31\textwidth]{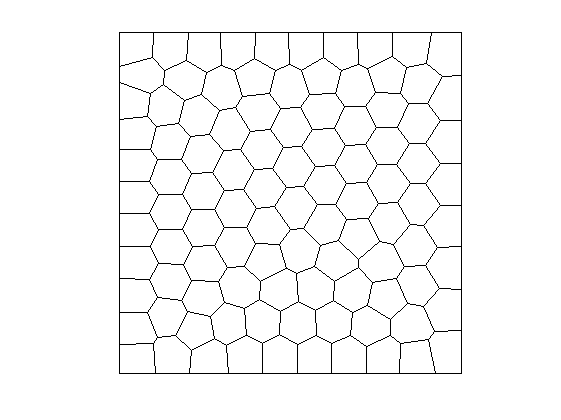}}
\subfigure [Level 2]{\label{fig: c2}
\includegraphics[height=0.25\textwidth,width=0.31\textwidth]{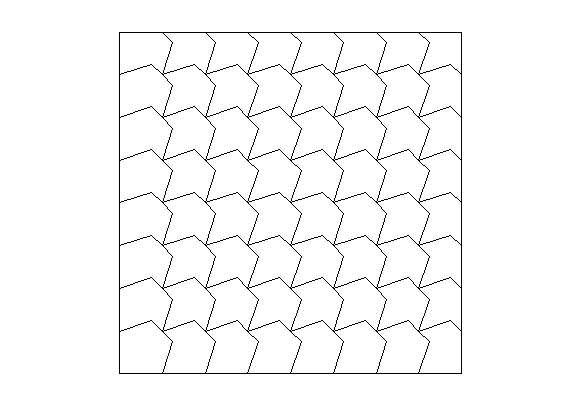}}
\caption{Level 1: Initial partitions $(a)-(c)$; Level 2: Partitions after one refinement  ($d$)-$(f)$.}\label{Pictu-polygonal}
\end{figure*}

In addition to computing $\3bar e_h\3bar$, $\|e_0\|$, $\|e_b\|_{{\cal E}_h}$ and $\|e_n\|_{\F_h}$,  more metrics are employed
\begin{equation*}
\begin{split}
 \ \displaystyle\|\nabla_{w,\F}e_b\|_{\F_h}
&=\Big(\sum_{T\in{\cal T}_{h}}h_T\|\nabla_{w,\F}e_b\|^2_{\pa T}\Big)^{1/2},\\
  \ \displaystyle\|\nabla(u-u_0)\|&=\Big(\sum_{T\in{\cal T}_{h}}\|\nabla(u-u_0)\|^2_{T}\Big)^{1/2}.
\end{split}
\end{equation*}

\noindent\underline{Test Example 1.} Table \ref{NE:polygonal-1} shows some numerical results when the exact solution is given by $u=\cos(x+1)\sin(2y-1)$ in the domain $\O=(0,1)^2$  on different types of polygonal partitions shown in Figure \ref{Pictu-polygonal}. For the uniform triangular partition and uniform rectangular partition, we can see from Table \ref{NE:polygonal-1} that the convergence rates for $\3bar e_h\3bar$, $\|e_0\|$, $\|e_b\|_{\E_h}$ are consistent with what our theory predicts, and the convergence rate for $\|e_n\|_{\F_h}$ is higher than the theoretical prediction of ${\cal O}(h)$. Moreover,  we   observe the convergence rates for  $\|\nabla_{w,\F}e_b\|_{\F_h}$ and $\|\nabla(u-u_0)\|$ are of order ${\cal O}(h^2)$ on the uniform triangular partition and uniform rectangular partition, for which the theory has not been developed in this paper. In addition, note that the theory established in previous sections does not cover the  polygonal partitions shown in Figure \ref{Pictu-polygonal}. However, we compute  the convergence rates in various norms on the polygonal partitions shown in Figure \ref{Pictu-polygonal} using the least-square methods \cite{CWY2014} and the corresponding convergence rates in various norms are illustrated in Table \ref{NE:polygonal-1}.

\begin{table}[h]
\caption{Numerical errors and convergence rates for the exact solution $u=\cos(x+1)\sin(2y-1)$ on different polygonal partitions in   $\O=(0,1)^2$.}\label{NE:polygonal-1}
\begin{tabular}{|p{0.7cm}p{1.48cm}p{1.48cm}p{1.48cm}p{1.48cm}p{1.58cm}p{1.80cm}|}
\hline
Level&$\3bare_h\3bar$&$\|e_0\|$&$\|e_b\|_{\E_h}$&$\|e_n\|_{\F_h}$&$\|\nabla_{w,\F}e_b\|_{\F_h}$&
$\|\nabla(u-u_0)\|$\\
\hline
&\text{Uniform~triangular~partition}&&&&&\\
\hline
1          &1.58E-01      &1.54E-03      &6.04E-04      &1.44E-02      &7.43E-03      &8.32E-03\\
2          &8.20E-02      &3.94E-04      &1.61E-04      &3.86E-03      &2.15E-03      &2.24E-03\\
3          &4.15E-02      &9.97E-05      &4.07E-05      &9.84E-04      &5.60E-04      &5.72E-04\\
4          &2.08E-02      &2.50E-05      &1.02E-05      &2.47E-04      &1.41E-04      &1.44E-04\\
5          &1.04E-02      &6.25E-06      &2.55E-06      &6.20E-05      &3.55E-05      &3.60E-05\\
\text{Rate}&1.00          &2.00          &2.00          &2.00          &2.00          &2.00\\
\hline
&\text{Uniform~rectangular~partition}&&&&&\\
\hline
1          &2.23E-01      &9.10E-04      &4.24E-04      &2.59E-02      &4.91E-03      &2.03E-02\\
2          &1.22E-01      &1.91E-04      &1.84E-04      &7.73E-03      &2.26E-03      &6.04E-03\\
3          &6.24E-02      &4.64E-05      &5.37E-05      &2.05E-03      &6.86E-04      &1.60E-03\\
4          &3.15E-02      &1.15E-05      &1.39E-05      &5.21E-04      &1.81E-04      &4.06E-04\\
5          &1.58E-02      &2.86E-06      &3.49E-06      &1.31E-04      &4.57E-05      &1.02E-04\\
\text{Rate}&1.00          &2.01          &1.99          &1.99          &1.98          &1.99\\
\hline
&\text{Randomised~quadrilateral~partition}&&&&&\\
\hline
1          &4.10E-01      &7.70E-03      &6.70E-03      &1.30E-01      &2.29E-02      &7.12E-02\\
2          &2.88E-01      &1.90E-03      &2.07E-03      &5.66E-02      &1.15E-02      &3.21E-02\\
3          &1.65E-01      &4.85E-04      &7.91E-04      &1.87E-02      &5.14E-03      &1.05E-02\\
4          &8.56E-02      &1.21E-04      &2.31E-04      &5.08E-03      &1.47E-03      &2.84E-03\\
5          &4.40E-02      &3.19E-05      &6.86E-05      &1.35E-03      &4.36E-04      &7.52E-04\\
\text{Rate}&0.99          &2.05          &1.84          &1.98          &1.85          &1.98\\
\hline
&\text{Hexagonal~partition}&&&&&\\
\hline
1          &4.82E-01      &5.82E-03      &7.97E-03      &1.07E-01      &1.08E-01      &7.38E-02\\
2          &3.23E-01      &1.08E-03      &1.64E-03      &4.03E-02      &4.57E-02      &2.84E-02\\
3          &1.75E-01      &2.38E-04      &4.30E-04      &1.25E-02      &1.21E-02      &8.69E-03\\
4          &1.03E-01      &4.53E-05      &9.46E-05      &3.51E-03      &5.04E-03      &2.43E-03\\
5          &5.30E-02      &8.70E-06      &1.89E-05      &8.92E-04      &1.37E-03      &6.15E-04\\
\text{Rate}&0.86          &2.38          &2.25          &1.90          &1.57          &1.91\\
\hline
&\text{Non-convex~octagonal~partition}&&&&&\\
\hline
1          &3.90E-01      &4.22E-03      &9.68E-03      &5.20E-02      &6.99E-02      &3.99E-02\\
2          &2.69E-01      &9.29E-04      &2.34E-03      &2.07E-02      &3.11E-02      &1.51E-02\\
3          &1.54E-01      &2.36E-04      &5.92E-04      &6.29E-03      &1.01E-02      &4.58E-03\\
4          &8.11E-02      &6.53E-05      &1.63E-04      &1.71E-03      &2.81E-03      &1.24E-03\\
5          &4.15E-02      &1.72E-05      &4.31E-05      &4.41E-04      &7.37E-04      &3.22E-04\\
\text{Rate}&0.94          &1.89          &1.89          &1.92          &1.89          &1.91\\
\hline
\end{tabular}
\end{table}

\noindent\underline{Test Example 2.} Table \ref{NE:polygonal-2} presents the numerical results on the uniform cubic partition in $\O=(0,1)^3$ for the exact solution $u=\exp(x+y+z)$. The convergence rates for $\3bare_h\3bar$, $\|e_0\|$ and $\|e_b\|_{\E_h}$ consist with our theory. Similar to Test Example 1, we can see  a super-convergence rate for $\|e_n\|_{\F_h}$ from Table \ref{NE:polygonal-2}. In addition,  Table \ref{NE:polygonal-2} presents the convergence rates for $\|\nabla_{w,\F}e_b\|_{\F_h}$ and $\|\nabla(u-u_0)\|$ for which  no theory is available to support.

\begin{table}[h]
\begin{center}
\centering
\caption{Numerical errors and  convergence rates for the exact solution $u=\exp(x+y+z)$ on the uniform cubic partition in $\O=(0,1)^3$.}\label{NE:polygonal-2}
\begin{tabular}{|p{1cm}p{1.48cm}p{1.58cm}p{1.58cm}p{1.58cm}p{1.88cm}p{0.8cm}|}
\hline
Level           &$\3bare_h\3bar$&Rate   &$\|e_0\|$ &Rate        &$\|e_b\|_{\E_h}$ &Rate\\
\hline
1               &2.68E-00  &~~-         &3.37E-01  &~~-         &1.88E-01  &~~-\\
2               &1.79E-00  &0.58        &3.49E-02  &3.27        &4.69E-02  &2.00\\
3               &9.85E-01  &0.86        &5.52E-03  &2.66        &1.24E-02  &1.92\\
4               &5.07E-01  &0.96        &1.10E-03  &2.33        &3.11E-03  &1.99\\
5               &2.55E-01  &0.99        &2.50E-04  &2.15        &7.67E-04  &2.02\\
\hline
Level           &$\|e_n\|_{\F_h}$ &Rate        &$\|\nabla_{w,\F}e_b\|_{\F_h}$&Rate&$\|\nabla(u-u_0)\|$&Rate\\
\hline
1               &8.05E-01  &~~-         &3.30E-01  &~~-         &6.62E-01  &~~-   \\
2               &3.15E-01  &1.35        &1.41E-01  &1.22        &2.58E-01  &1.36\\
3               &8.93E-02  &1.82        &5.26E-02  &1.42        &8.04E-02  &1.68\\
4               &2.30E-02  &1.96        &1.53E-02  &1.79        &2.18E-02  &1.88\\
5               &5.76E-03  &2.00        &4.00E-03  &1.93        &5.61E-03  &1.96\\
\hline
\end{tabular}
\end{center}
\end{table}

\noindent\underline{Test Example 3.} Table \ref{NE:polygonal-3} illustrates the numerical performance   on the polygonal partitions shown in Figure \ref{Pictu-polygonal} for a low regularity solution given by $u=r^{5/3}\sin(\frac{5}{3}\theta)$, where $r=\sqrt{x^2+y^2}$ and $\theta=\arctan(y/x)$. It is easy to check $u\in H^{8/3-\varepsilon}(\O)$ for arbitrary small $\varepsilon>0$ does not satisfy the regularity assumption $u\in H^4(\O)$. We observe from these numerical results that on the uniform triangular partition and uniform rectangular partition, the convergence rates for $\3bar e_h\3bar$, $\|e_0\|$, $\|e_b\|_{\E_h}$, $\|e_n\|_{\F_h}$, $\|\nabla_{w,\F}e_b\|_{\F_h}$, $\|\nabla(u-u_0)\|$ are  of orders ${\cal O}(h^{2/3})$, ${\cal O}(h^2)$, ${\cal O}(h^2)$, ${\cal O}(h^{5/3})$, ${\cal O}(h^{5/3})$, ${\cal O}(h^{5/3})$, respectively. Moreover, the numerical performance of the WG solution on the polygonal partitions is demonstrated in Table \ref{NE:polygonal-3}.

\begin{table}[h]
\caption{Numerical errors and convergence rates for the  exact solution $u=r^{5/3}\sin(\frac{5}{3}\theta)$ on the polygonal partitions in $\O=(0,1)^2$.}\label{NE:polygonal-3}
\begin{tabular}{|p{0.7cm}p{1.48cm}p{1.48cm}p{1.48cm}p{1.48cm}p{1.58cm}p{1.80cm}|}
\hline
Level&$\3bare_h\3bar$&$\|e_0\|$&$\|e_b\|_{\E_h}$&$\|e_n\|_{\F_h}$&$\|\nabla_{w,\F}e_b\|_{\F_h}$&
$\|\nabla(u-u_0)\|$\\
\hline
&\text{Uniform~triangular~partition}&&&&&\\
\hline
1          &2.99E-02      &7.31E-04      &2.82E-04      &2.02E-03      &3.55E-03      &2.15E-03\\
2          &2.00E-02      &1.86E-04      &7.44E-05      &6.61E-04      &1.20E-03      &7.02E-04\\
3          &1.31E-02      &4.53E-05      &1.83E-05      &2.12E-04      &3.90E-04      &2.25E-04\\
4          &8.39E-03      &1.10E-05      &4.48E-06      &6.74E-05      &1.25E-04      &7.16E-05\\
5          &5.35E-03      &2.71E-06      &1.11E-06      &2.14E-05      &3.97E-05      &2.27E-05\\
\text{Rate}&0.65          &2.02          &2.02          &1.66          &1.65          &1.66\\
\hline
&\text{Uniform~rectangular~partition}&&&&&\\
\hline
1          &1.30E-01      &1.74E-03      &2.20E-03      &1.57E-02      &1.73E-02      &1.22E-02\\
2          &8.89E-02      &5.27E-04      &6.53E-04      &5.62E-03      &6.43E-03      &4.31E-03\\
3          &5.80E-02      &1.36E-04      &1.68E-04      &1.83E-03      &2.15E-03      &1.40E-03\\
4          &3.73E-02      &3.37E-05      &4.14E-05      &5.85E-04      &6.99E-04      &4.45E-04\\
5          &2.38E-02      &8.35E-06      &1.02E-05      &1.86E-04      &2.24E-04      &1.41E-04\\
\text{Rate}&0.65          &2.02          &2.02          &1.66          &1.64          &1.66\\
\hline
&\text{Randomised~quadrilateral~partition}&&&&&\\
\hline
1          &1.59E-01      &5.47E-03      &1.75E-02      &3.82E-02      &3.71E-02      &2.74E-02\\
2          &1.32E-01      &2.26E-03      &7.64E-03      &1.79E-02      &2.03E-02      &1.24E-02\\
3          &9.09E-02      &7.20E-04      &2.53E-03      &6.49E-03      &7.74E-03      &4.26E-03\\
4          &5.95E-02      &1.77E-04      &6.41E-04      &2.16E-03      &2.55E-03      &1.39E-03\\
5          &3.85E-02      &5.25E-05      &1.95E-04      &6.89E-04      &8.77E-04      &4.39E-04\\
\text{Rate}&0.65          &1.97          &1.93          &1.69          &1.64          &1.71\\
\hline
&\text{Hexagonal~partition}&&&&&\\
\hline
1          &2.45E-01      &3.22E-03      &9.78E-03      &2.30E-02      &1.07E-01      &1.89E-02\\
2          &2.02E-01      &1.34E-03      &3.36E-03      &8.79E-03      &5.30E-02      &7.82E-03\\
3          &1.30E-01      &4.49E-04      &1.02E-03      &3.60E-03      &1.72E-02      &3.28E-03\\
4          &7.75E-02      &8.34E-05      &1.86E-04      &1.17E-03      &5.30E-03      &1.07E-03\\
5          &5.27E-02      &1.99E-05      &4.41E-05      &3.89E-04      &1.76E-03      &3.17E-04\\
\text{Rate}&0.65          &2.25          &2.26          &1.61          &1.64          &1.68\\
\hline
&\text{Non-convex~octagonal~partition}&&&&&\\
\hline
1          &1.43E-01      &1.53E-03      &4.87E-03      &1.41E-02      &3.02E-02      &1.86E-02\\
2          &1.10E-01      &6.96E-04      &1.82E-03      &5.89E-03      &1.39E-02      &6.97E-03\\
3          &7.59E-02      &2.42E-04      &6.06E-04      &2.14E-03      &5.11E-03      &2.45E-03\\
4          &5.02E-02      &6.84E-05      &1.71E-04      &7.20E-04      &1.72E-03      &8.15E-04\\
5          &3.25E-02      &1.75E-05      &4.39E-05      &2.36E-04      &5.60E-04      &2.64E-04\\
\text{Rate}&0.61          &1.89          &1.89          &1.59          &1.59          &1.61\\
\hline
\end{tabular}
\end{table}

\noindent\underline{Test Example 4.} Table \ref{NE:polygonal-4} demonstrates the numerical performance
on the uniform cubic partition in $\O=(0,1)^3$ for a low regularity solution given by $u=r^{3/2}\sin(\frac{3}{2}\theta)$, where $r=\sqrt{x^2+y^2}$ and $\theta=\arctan(y/x)$. The exact solution satisfies $u\in H^{5/2-\varepsilon}(\O)$ for arbitrary small $\varepsilon>0$. We observe that the numerical errors $\3bare_h\3bar$, $\|e_0\|$, $\|e_b\|_{\E_h}$, $\|e_n\|_{\F_h}$, $\|\nabla_{w,\F}e_b\|_{\F_h}$, $\|\nabla(u-u_0)\|$ converge at the rates of ${\cal O}(h^{1/2})$, ${\cal O}(h^2)$, ${\cal O}(h^2)$, ${\cal O}(h^{3/2})$, ${\cal O}(h^{3/2})$, ${\cal O}(h^{3/2})$, respectively. Therefore,  we conclude that the numerical performance of the WG method for the model equation \eqref{model-problem} with the low regularity solution is good although the corresponding mathematical theory has not been established in our paper.

\begin{table}[h]
\begin{center}
\caption{Numerical errors and convergence rates for the  exact solution $u=r^{3/2}\sin(\frac{3}{2}\theta)$ on the uniform cubic partition in $\O=(0,1)^3$.}\label{NE:polygonal-4}
\begin{tabular}{|p{1cm}p{1.48cm}p{1.58cm}p{1.58cm}p{1.58cm}p{1.88cm}p{0.8cm}|}
\hline
Level           &$\3bare_h\3bar$&Rate   &$\|e_0\|$ &Rate        &$\|e_b\|_{\E_h}$ &Rate\\
\hline
1               &1.43E-01  &~~-         &3.62E-03  &~~-         &1.26E-02  &~~-\\
2               &1.32E-01  &0.12        &1.52E-03  &1.26        &3.78E-03  &1.74\\
3               &9.65E-02  &0.45        &5.04E-04  &1.59        &1.04E-03  &1.86\\
4               &6.84E-02  &0.50        &1.39E-04  &1.86        &2.70E-04  &1.95\\
\hline
Level           &$\|e_n\|_{\F_h}$&Rate         &$\|\nabla_{w,\F}e_b\|_{\F_h}$&Rate&$\|\nabla(u-u_0)\|$&Rate\\
\hline
1               &3.27E-02  &~~-         &4.69E-02  &~~-         &4.13E-02  &~~-\\
2               &1.58E-02  &1.05        &2.55E-02  &0.88        &2.03E-02  &1.03\\
3               &5.65E-03  &1.48        &1.06E-02  &1.26        &8.09E-03  &1.32\\
4               &1.90E-03  &1.57        &3.99E-03  &1.41        &2.97E-03  &1.45\\
\hline
\end{tabular}
\end{center}
\end{table}

\end{document}